\documentclass[11pt]{article}
\usepackage{amssymb}
\usepackage{amsfonts}
\usepackage{amsmath}
\usepackage{mathrsfs}
\usepackage{graphicx}
\usepackage{amsbsy}
\usepackage{theorem}
\usepackage{color}
\usepackage{tikz}
\usepackage{bbm}
\usepackage[normalem]{ulem}
\newtheorem{theorem}{Theorem}[section]
\newtheorem{lemma}[theorem]{Lemma}

\def\thetheorem{\thesection.\arabic{theorem}}
\def\thesection{\arabic{section}}

\def\theequation {\thesection.\arabic{equation}}
\def\beq{\begin{equation}\displaystyle}
\def\eeq{\end{equation}}
\def\bel{\begin{equation} \displaystyle \begin{array}{l} }
\def\eel{\end{array} \end{equation} }
\def\bell{\begin{equation} \displaystyle \begin{array}{ll}  }
\def\eell{\end{array} \end{equation} }

\def\bea{\begin{eqnarray}}
\def\eea{\end{eqnarray} }
\def\bean{\begin{eqnarray*}}
\def\eean{\end{eqnarray*} }
\newenvironment{proof}{\noindent{\bf Proof.~}}
{{\mbox{}\hfill {\small \fbox{}}\\}}
\catcode`@=11
\renewcommand\appendix{\bigskip {\noindent \Large \bf Appendix}
  \setcounter{section}{0}%
  \setcounter{subsection}{0}%
\setcounter{equation}{0}%
\setcounter{theorem}{0}%
\def\thetheorem{A.\arabic{theorem}}
\def\theequation {A.\arabic{equation}}}
\catcode`@=12

\def\RR{\mathbb{R}}
\def\VV{\mathbb{V}}
\def\NN{\mathbb{N}}
\def\ZZ{\mathbb{Z}}
\def\eps{\varepsilon}
\def\pa{\partial}
\def\DT{\Delta t}
\def\DX{\Delta x}
\def\DY{\Delta y}

\title{Numerical scheme for kinetic transport equation with internal state\thanks{This work was supported by Japan-France Integrated action Program (SAKURA), Grant number JPJSBP120193219.}}

\author{Nicolas Vauchelet\thanks{Universit\'e Sorbonne Paris Nord, Laboratoire Analyse, G\'eom\'etrie et Applications, LAGA, CNRS UMR 7539, F-93430, Villetaneuse, France.
    (\texttt{email : vauchelet@math.univ-paris13.fr}).}
  \and Shugo Yasuda\thanks{University of Hyogo, Graduate School of Simulation Studies, Kobe 650-0047, Japan 
  (\texttt{email : yasuda@sim.u-hyogo.ac.jp}).}
}

\begin{document}

\maketitle

\begin{abstract}
  We investigate the numerical discretization of a two-stream kinetic system with an internal state, such system has been introduced to model the motion of cells by chemotaxis. This internal state models the intracellular methylation level. It adds a variable in the mathematical model, which makes it more challenging to simulate numerically.
  Moreover, it has been shown that the macroscopic or mesoscopic quantities computed from this system converge to the Keller-Segel system at diffusive scaling or to the velocity-jump kinetic system for chemotaxis at hyperbolic scaling.
  Then we pay attention to propose numerical schemes uniformly accurate with respect to the scaling parameter. We show that these schemes converge to some limiting schemes which are consistent with the limiting macroscopic or kinetic system. This study is illustrated with some numerical simulations and comparisons with Monte Carlo simulations.
\end{abstract}

\medskip
\textbf{Keywords}:
  Asymptotic-preserving scheme, chemotaxis, kinetic-transport model with internal state, well-balanced scheme.

\medskip
\textbf{AMS Classification}:
  65M08, 65N08, 65M12, 92C17.

\section{Introduction}

Kinetic equations are by now widely used to model the motion of bacteria by so-called 'run-and-tumble' process, i.e. an alternance of forward-moving runs with reorienting tumbles. In a simple mathematical model, bacteria respond to a given external chemical signal only by modulating their probability to tumble.
Then, the motion of bacteria may be described by the dynamics of their distribution function, which corresponds to the probability to find a bacteria at time $t$ with position $x$ and velocity $v$. This mathematical approach has been introduced in \cite{ODA} and has attracted the attention of many mathematical works, in particular it has been shown that Keller-Segel equations may be derived at diffusive scaling (see e.g. \cite{OH,ErbanOthmer04,DS,HP} and references therein).

Actually, the response of bacteria to signal changes is governed by a sophisticated chemotactic signal transduction pathway. It involves a rapid response of the cell to the external signal change called 'excitation', and a slow 'adaptation' which allows the cell to subtract out the background signal.
In order to describe the multiscale mechanism between the intracellular chemo-sensory system and the individual and collective motions of bacteria,
 more elaborated kinetic models have been proposed.
In \cite{ErbanOthmer04, ErbanOthmer07}, a cell-based model which incorporates a linear cartoon description of the excitation and adaptation response of a cell has been introduced.
The signal transduction pathway has been studied in e.g. \cite{hazel,OXX}. We refer to \cite{OXX,tu,ZSOT,X} where the case of bacteria {\it E. coli} and the link between cell-based models and continuum models have been extensively studied.

In this paper, we consider the simplest description of the biochemical pathways, in which a single additional variable $m \geq 0$, which represents the intracellular methylation level, is used.
The methylation has an equilibrium level $M$, depending on time $t$ and position $x$.
Using $F(m,M)$ as the adaptation rate, the intracellular adaptation dynamics is governed by a differential equation 
\begin{equation*}
\frac{dm}{dt} = F(m,M(t,x)).
\end{equation*}
In order to model the dynamics of a population of bacteria, the following pathway-based kinetic-transport model has been proposed in \cite{ErbanOthmer04,SWOT}.
Let $q(t,x,v,m)$ be the probability density function of bacteria at time $t$, position $x\in\RR^d$, moving at velocity $v\in \VV$, where $\VV$ is a bounded domain of $\RR^d$, and methylation level $m >0$. We have
\begin{equation}\label{eq:kineticm} \begin{cases}
\pa_t q+v\cdot\nabla_{x} q+\pa_m [F(m,M) q] = {\mathcal Q}[m,M](q),
\\[5pt]
q(x,v,m=0,t)=0.
\end{cases}
\end{equation}
The 'run' phase is modelled by the transport term in $x$, whereas the evolution of the methylation level is modelled by the $m$-derivative.
The 'tumble' phase is described in the right hand side by the tumbling operator ${\mathcal Q}[m,M](q)$. It is given by
\begin{equation}\label{eq:Q}
{\mathcal Q}[m,M](q)=\frac{1}{\|\VV\|} \int_\VV \left[
\lambda (m,M,v,v')q(t,x,v',m)  - 
 \lambda (m,M,v',v) q(t,x,v,m) \right] \,dv',
\end{equation}
where $\lambda(m,M,v,v')$ denotes the methylation dependent tumbling frequency from $v'$ to $v$.
Departing from such pathway-based kinetic equation several authors (see \cite{ErbanOthmer04, ErbanOthmer07, DS,XO,STY,X,ST,PTY}) developed the asymptotic theory which allows to recover, in the diffusion limit, macroscopic equations as the Keller-Segel (KS) or the flux-limited Keller-Segel (FLKS) system which governs the dynamics of the macroscopic quantities depending only on $t$ and $x$.
In \cite{JMB}, the authors show that, at hyperbolic scaling, the above pathway-based kinetic equation converges to the classical velocity jump kinetic equation for mesoscopic quantities depending on $t$, $x$, and $v$.

For \textit{E. coli} chemotaxis, it may be assumed that the tumbling frequency $\lambda$ in \eqref{eq:Q} depends only on $M-m$ \cite{SWOT}~: $\lambda(m,M,v,v') = \lambda_0 \Lambda(\frac{M-m}{\delta})$, where $\lambda_0>0$ is the mean tumbling frequency and $\delta>0$ is the stiffness of the chemotactic response.
Moreover, we consider the linear model
\begin{equation*}
\frac{dm}{dt} = F(m,M) = \frac{1}{\tau} (M - m),
\end{equation*}
where $\tau>0$ is the characteristic time of adaptation.
The equilibrium level $M$ is a function of extra-cellular chemical concentration and is assumed to be independent of time. Usually, it has a logarithmic dependancy as it has been experimentally evidenced for {\it E. coli} in \cite{kalinin}, which allows to consider that $\nabla_x M$ is uniform in space in exponential environment. We denote $G\in \RR^d$ the constant vector $G=\nabla_x M$.
Then, we introduce the new variable $y=\frac{M(x)-m}{\delta}$, and the unknown $p(t,x,v,y)=q(t,x,v,m)$.
After this change of variable and introducing a time scaling parameter $\sigma$, equation \eqref{eq:kineticm} rewrites (see \cite{JMB,PTY})
\begin{equation}\label{eq_kinetic_internal}
\sigma\partial_t p(t,x,v,y)+v\cdot\nabla_xp+\partial_y\left((v\cdot \frac{G}{\delta}-\frac{y}{\tau})p\right) = \lambda_0\Lambda(y)(\langle p \rangle - p ).
\end{equation}
%
An asymptotic analysis of Eq.~(\ref{eq_kinetic_internal}) depending on the scalings of $\lambda_0^{-1}$, $\delta$, $\sigma$, and $\tau$ may lead to various models.
In \cite{PTY}, several scalings have been investigated depending on the relative magnitude order of $\lambda_0^{-1}=\eps$, $\delta=\tau$, and $\sigma$.
In particular, the authors recover macroscopic flux-limited Keller-Segel models or Keller-Segel models as $\eps\to 0$ when $\delta=\tau=\eps$ and $\sigma=\eps^2$, when $\delta=\tau=\eps^2$, $\sigma=\eps$, and when $\delta=\tau=1$, $\sigma=\eps$.
In \cite{JMB}, the authors consider the hyperbolic case $\lambda_0^{-1}=1$ and $\delta=\tau=\eps$ for which the limiting model is the kinetic 'run-and-tumble' model.
Let us mention that these limits are singular in the sense that the distribution function converges to Dirac deltas.

In order to further illustrate those asymptotic behaviors numerically and elucidate the complicated multiscale mechanism between the intra-cellular pathway dynamics and the collective motions of bacteria, one needs an accurate and efficient numerical method which can address the asymptotic limits.
This paper is devoted to the numerical discretization and simulation of such kinetic system with internal variable \eqref{eq_kinetic_internal}.
In order to simplify the study and facilitate the presentation of the computations, we consider the problem in one dimensional spatial domain and assume that $v\in \{-1,1\}$. We denote by $p^{\eps,+}(t,x,y)$ and $p^{\eps,-}(t,x,y)$ the unknowns, depending on $\eps$, of our problem. In this one dimensional setting with $\delta=\tau$, equation \eqref{eq_kinetic_internal} rewrites
\begin{equation}
  \label{twovelo}
  \sigma \pa_t p^{\eps,\pm} \pm \pa_x p^{\eps,\pm} + \frac{1}{\tau} \pa_y((\pm G -y) p^{\eps,\pm}) = \pm \lambda_0 \frac{\Lambda(y)}{2}(p^{\eps,-} - p^{\eps,+}).
\end{equation}
This system is complemented with some initial data $p^{\eps,\pm}(t=0,y)$.
In this work, we would like to propose numerical schemes for \eqref{twovelo} which are able to deal with the diffusive limit $(\sigma = \lambda_0^{-1} = \eps\ll 1=\tau)$ and the hyperbolic limit $(\sigma=\lambda_0=1,\tau=\eps\ll 1)$.
This is the so-called asymptotic preserving (AP) property \cite{GT,Jin}.
There are several challenging issues in doing so. As already mentionned, one difficulty is due to the extra variable $y$, then usual AP numerical schemes for kinetic equations may not be directly applied. Another difficulty is that the distribution function has a singular limit when $\eps\to 0$ since it converges to Dirac deltas (see \cite{JMB,PTY}).
To overcome these difficulties, we follow the approach in \cite{book}, which proposes accurate numerical schemes preserving stationary states, and we couple it with a projection step which allows us to capture the Dirac deltas.

The paper is organized as follows. 
In Section \ref{sec:diffscal}, we investigate the numerical discretization of system \eqref{twovelo} at diffusive scaling. We propose a numerical scheme having the so-called asymptotic preserving property. Section \ref{sec:hyperbolic} deals with the hyperbolic scaling. Although we follow the same idea as for the diffusive scaling, the scheme should be adapted to this scaling, as it is emphasized in \S\ref{subsec:wrong}.
Numerical illustrations are proposed in Section \ref{Sec_Num} and compared with a Monte Carlo (MC) scheme.
This paper ends with a conclusion.
Finally, an appendix gathers a useful technical computation and a presentation of the MC scheme used for comparison.

\section{Diffusive scaling}
\subsection{System at diffusive scaling}\label{sec:diffscal}

We first consider system \eqref{twovelo} at the diffusive limit $\sigma = \lambda_0^{-1} = \eps, \tau=1$.
We denote $p^{\eps,+}(t,x,y)$, resp. $p^{\eps,-}(t,x,y)$, the distribution function of bacteria at time $t>0$, with internal state $y$, going to the positive, resp. negative, direction. Their dynamics is governed by the system~:
\begin{equation}
  \label{twovelo2}
  \eps^2 \pa_t p^{\eps,\pm} \pm \eps \pa_x p^{\eps,\pm} + \eps \pa_y((\pm G -y) p^{\eps,\pm}) = \pm \frac{\Lambda(y)}{2}(p^{\eps,-} - p^{\eps,+}).
\end{equation}
We first notice that it has been proved in \cite{PTY} that when the support of initial data ${p}^{\eps,\pm}(0,\cdot)$ is included into $[-|G|,|G|]$, then for any $t>0$ the support of ${p}^{\eps,\pm}(t,\cdot)$ is also included into $[-|G|,|G|]$. Thus we will assume that the support in $y$ of the initial data is included into $[-|G|,|G|]$ such that the domain of the internal variable $y$ is $[-|G|,|G|]$.

Before focusing on the numerical discretisation of \eqref{twovelo2}, we first explain how to perform formally the diffusive limit $\eps\to 0$ into \eqref{twovelo2}.
Let us assume formally that $p^{\eps,\pm}$ admit a Hilbert expansion $p^{\eps,\pm} = p_0^\pm + \eps p^\pm_1 + O(\eps^2)$.
Letting $\eps \to 0$ in \eqref{twovelo2}, we first deduce that $p_0^+ = p_0^- = p_0$.
Then, injecting the Hilbert expansion into \eqref{twovelo2} and identifying the term at order 1 in $\eps$, we get
\[
\pm \pa_x p_0 + \pa_y((\pm G-y) p_0 ) = \pm \frac{\Lambda(y)}{2} (p^-_1-p^+_1).
\]
Adding these two equations, we deduce $\pa_y (y p_0)=0$. Hence $p_0(t,x,y) = \bar{p}_0(t,x) \delta_{y=0}$ for some $\bar{p}_0$. We are left with the relation
\[
\pa_x p_0 + \pa_y( G p_0 ) = \frac{\Lambda(y)}{2} (p^-_1-p^+_1).
\]
Therefore, injecting the relation $p_0(t,x,y) = \bar{p}_0(t,x) \delta_{y=0}$, we get
\[
p^-_1-p^+_1 = \frac{2}{\Lambda(y)} \pa_x \bar{p}_0 \delta_{y=0} + \frac{2G\bar{p}_0}{\Lambda(y)} \delta'_{y=0}.
\]
Integrating over $y$, we deduce
\begin{equation}
  \label{eq:p1}
\int_\RR (p^-_1-p^+_1)\,dy  = \frac{2}{\Lambda(0)} \pa_x \bar{p}_0 + 2G\bar{p}_0 \frac{\Lambda'(0)}{\Lambda(0)^2}.
\end{equation}
Moreover, adding the two equations in \eqref{twovelo2} and letting formally $\eps\to 0$ leads to
\[
2 \pa_t p_0 + \pa_x(p_1^+ - p_1^-) + \pa_y(G(p_1^+-p_1^-) - y (p_1^+ + p_1^- )) = 0.
\]
Integrating over $y$ and injecting \eqref{eq:p1}, we get
\begin{equation}
  \label{eq:KS}
\pa_t \bar{p}_0 - \pa_x\left(\frac{1}{\Lambda(0)} \pa_x \bar{p}_0 + G\bar{p}_0 \frac{\Lambda'(0)}{\Lambda(0)^2}\right) = 0,
\end{equation}
which is the Keller-Segel equation for the density $\bar{p}_0$.
This formal computation has been done rigorously in \cite{PTY}.

\subsection{A consistent uniform numerical scheme}

From above computations, we observe that a difficulty in the design of the numerical scheme is the emergence of Dirac deltas at the diffusive limit. Then, an uniform numerical scheme should be able to captur this Dirac deltas.
In order to overcome this difficulty, we consider an approach consisting in two steps:
\begin{itemize}
\item In a first step we solve
  \begin{equation}
    \label{eq:diff1}
    \pa_t p^\pm - \frac{1}{\eps} \pa_y (yp^\pm) = 0.
  \end{equation}
  When $\eps\to 0$ this step may be seen as a projection step onto the set of Dirac deltas in $y=0$ which is the limiting set. Indeed taking $\eps\to 0$ into this equation leads formally to $\pa_y(yp^\pm)=0$ which implies that $p^\pm$ are Dirac deltas in $y=0$. A similar idea has also been used in \cite{AML}.
\item In a second step we consider an uniform discretization of
  \begin{equation}
    \label{eq:diff2}
    \eps\pa_t p^\pm \pm \pa_x p^\pm \pm G \pa_y p^\pm = \pm \frac{\Lambda(y)}{2\eps}(p^- - p^+).
  \end{equation}
\end{itemize}

More precisely, let us consider a cartesian grid $x_i=i\Delta x$ for $i\in \ZZ$ and $y_k = k\Delta y$ for $k=-K,\ldots,K$ such that $K\Delta y = |G|$. In this paper, to simplify the computations, we assume that $\Delta y = |G| \Delta x$, unless otherwise stated.

In the first step, we discretize \eqref{eq:diff1} by an implicit upwind scheme
\begin{subequations}
  \label{upwdiff}
  \begin{align}
    \label{upwdiff1}
    & p_{i,k}^{\pm,n+\frac 12}  = p_{i,k}^{\pm,n} - \frac{\DT}{\eps\DY}\Big(J_{i,k+\frac 12}^{\pm,n+\frac 12} - J_{i,k-\frac 12}^{\pm,n+\frac 12}\Big), \\
    \label{upwdiff2}
    & J_{i,k+\frac 12}^{\pm,n+\frac 12} = (-y_k)^+ p_{i,k}^{\pm,n+\frac 12} - (-y_{k+1})^- p_{i,k+1}^{\pm,n+\frac 12}.
  \end{align}
\end{subequations}
In this scheme, we use the usual notation $u^+ = \max\{0,u\}$ for the positive part, and $u^- = \max\{0,-u\}$ for the negative part.
We impose no-flux boundary conditions at the boundary~: $J_{i,K+\frac 12}^{\pm,n+\frac 12} = J_{i,-K-\frac 12}^{\pm,n+\frac 12} = 0$, for any $i$.

In the second step, we discretize \eqref{eq:diff2} thanks to a well-balanced asymptotic preserving scheme. However, this equation may be considered as a two-dimensional equation with the two directions $x$ and $y$. In order to be consistent with the asymptotic limit, it is not a good idea to split the directions as it is illustrated in \S\ref{subsec:wrong}.
Truly two-dimensional well-balanced schemes are not much developed yet. We mention the recent paper \cite{BG}, and \cite{Sinum_2D} in the particular case of radiative transfer equation.
However, the case at hand can be seen as a one dimension problem by introducing a new variable.
Indeed, let us denote $q^\pm(t,z) = p^\pm(t,x+z,y+zG)$, then \eqref{eq:diff2} rewrites
\[
\eps \pa_t q^\pm \pm \pa_z q^\pm = \pm \frac{\Lambda(y+zG)}{2\eps} (p^- - p^+).
\]
This latter equation can be discretized following the idea in \cite{GT,Numerische}.
We find the numerical scheme
\begin{align*}
  & q_{j}^{+,n+1} = q_j^{+,n+\frac 12} - \frac{\DT}{\eps\DX} (q_j^{+,n+1} - q_{j}^{-,n+1}) + \frac{2\DT G}{\DX(2\eps G + \bar{\Lambda}_{k-\frac 12})}(q_{j-1}^{+,n+\frac 12} - q_{j}^{-,n+\frac 12})  \\
  & q_{j-1}^{-,n+1} = q_{j-1}^{-,n+\frac 12} + \frac{\DT}{\eps\DX} (q_{j-1}^{+,n+1} - q_{j-1}^{-,n+1}) + \frac{2\DT G}{\DX(2\eps G + \bar{\Lambda}_{k-\frac 12})}(q_{j}^{-,n+\frac 12} - q_{j-1}^{+,n+\frac 12}),
\end{align*}
where $\bar{\Lambda}_{k-\frac 12}=\int_{y_{k-1}}^{y_k} \Lambda(y)\,dy$. 
Coming back to the notation $p^\pm$, it gives, when $G>0$,
\begin{subequations}
  \label{eq:pdiff}
  \begin{align}
    \label{eq:pplusdiff}
    p_{i,k}^{+,n+1} =\ & p_{i,k}^{+,n+\frac 12} - \frac{\DT}{\eps\DX} (p_{i,k}^{+,n+1} - p_{i,k}^{-,n+1}) \\
    & + \frac{2\DT G}{\DX(2\eps G + \bar{\Lambda}_{k-\frac 12})}(p_{i-1,k-1}^{+,n+\frac 12} - p_{i,k}^{-,n+\frac 12}) \nonumber  \\
    \label{eq:pmoinsdiff}
    p_{i-1,k-1}^{-,n+1} =\ & p_{i-1,k-1}^{-,n+\frac 12} + \frac{\DT}{\eps\DX} (p_{i-1,k-1}^{+,n+1} - p_{i-1,k-1}^{-,n+1}) \\
    & + \frac{2\DT G}{\DX(2\eps G + \bar{\Lambda}_{k-\frac 12})}(p_{i,k}^{-,n+\frac 12} - p_{i-1,k-1}^{+,n+\frac 12}). \nonumber
  \end{align}
\end{subequations}
When $G<0$, we have
\begin{subequations}
  \label{eq:pdiff2}
  \begin{align}
    \label{eq:pplusdiff2}
    p_{i,k}^{+,n+1} =\ & p_{i,k}^{+,n+\frac 12} - \frac{\DT}{\eps\DX} (p_{i,k}^{+,n+1} - p_{i,k}^{-,n+1})  \\
    & + \frac{2\DT G}{\DX(2\eps G + \bar{\Lambda}_{k+\frac 12})}(p_{i-1,k+1}^{+,n+\frac 12} - p_{i,k}^{-,n+\frac 12}) \nonumber  \\
    \label{eq:pmoinsdiff2}
    p_{i-1,k+1}^{-,n+1} =\ & p_{i-1,k+1}^{-,n+\frac 12} + \frac{\DT}{\eps\DX} (p_{i-1,k+1}^{+,n+1} - p_{i-1,k+1}^{-,n+1}) \\
    & + \frac{2\DT G}{\DX(2\eps G + \bar{\Lambda}_{k+\frac 12})}(p_{i,k}^{-,n+\frac 12} - p_{i-1,k+1}^{+,n+\frac 12}). \nonumber
  \end{align}
\end{subequations}
Finally, the resulting numerical scheme is given by \eqref{upwdiff}--\eqref{eq:pdiff} when $G>0$, or \eqref{upwdiff}--\eqref{eq:pdiff2} when $G<0$.
We mention that \eqref{eq:pdiff} and \eqref{eq:pdiff2} are implicit but can be easily solved by inverting a 2-by-2 matrix.

The following Lemma gives some properties for this scheme:
\begin{lemma}
  Let us assume that $\DY = |G| \DX$, and that $\Lambda\geq \Lambda_{\text{min}} > 0$.
  Let $(p_{i,k}^{\pm,n})_{i,k,n}$ be a sequence defined thanks to scheme \eqref{upwdiff}--\eqref{eq:pdiff} when $G>0$, or \eqref{upwdiff}--\eqref{eq:pdiff2} when $G<0$.
  Then, under the condition
  \begin{equation}\label{CFL}
    \DT \leq \frac 12 \Lambda_{\text{min}} \DX^2,
  \end{equation}
  the scheme is positive.
\end{lemma}

\begin{proof}
  It is well-known that the implicit upwind scheme \eqref{upwdiff} is positive.
  For the second step, we perform the computations for $G>0$, the case $G<0$ being similar.
  From \eqref{eq:pdiff}, we deduce by inverting the system
  \begin{align*}
    \left(1+\frac{2\DT}{\eps \DX}\right) p_{i,k}^{+,n+1} =\ & \left(1+\frac{\DT}{\eps\DX}\right) p_{i,k}^{+,n+\frac 12} + \frac{\DT}{\eps \DX} p_{i,k}^{-,n+\frac 12}  \\
    & + \frac{2\DT(\DT+\eps \DX) G}{\eps \DX^2 (2\eps G + \bar{\Lambda}_{k-\frac 12})}(p_{i-1,k-1}^{+,n+\frac 12} - p_{i,k}^{-,n+\frac 12})  \\
    & + \frac{2\DT^2 G}{\eps \DX^2 (2\eps G + \bar{\Lambda}_{k+\frac 12})}(p_{i+1,k+1}^{-,n+\frac 12} - p_{i,k}^{+,n+\frac 12}),
  \end{align*}
  and
  \begin{align*}
    \left(1+\frac{2\DT}{\eps \DX}\right) p_{i,k}^{-,n+1} =\ & \left(1+\frac{\DT}{\eps\DX}\right) p_{i,k}^{-,n+\frac 12} + \frac{\DT}{\eps \DX} p_{i,k}^{+,n+\frac 12}  \\
    & + \frac{2\DT^2 G}{\eps \DX^2 (2\eps G + \bar{\Lambda}_{k-\frac 12})}(p_{i-1,k-1}^{+,n+\frac 12} - p_{i,k}^{-,n+\frac 12}) \\
    & + \frac{2\DT(\DT+\eps \DX) G}{\eps \DX^2 (2\eps G + \bar{\Lambda}_{k+\frac 12})}(p_{i+1,k+1}^{+,n+\frac 12} - p_{i,k}^{-,n+\frac 12}).
  \end{align*}
  Then, this scheme is positive provided all coefficients are positive, which is equivalent to, for all $k\in \{-K,\ldots,K\}$,
  \[
  \frac{2\DT G}{\DX (2\eps G + \bar{\Lambda}_{k+\frac 12})} \leq \min\left\{\frac{\DT}{\DT+\eps \DX},\frac{\DT+\eps \DX}{\DT}\right\} = \frac{\DT}{\DT+\eps \DX}.
  \]
  Since we have by definition $\bar{\Lambda}_{k+\frac 12}= \int_{y_k}^{y_{k+1}} \Lambda(y)\,dy \geq \Lambda_{\text{min}} \DY$, then the above inequality is satisfied for any $k$ provided
  \[
  \frac{2 G}{\DX (2\eps G + \Lambda_{\text{min}} \DY)} \leq \frac{1}{\DT+\eps \DX},
  \]
  which is equivalent to \eqref{CFL} by recalling that $\DY=G\DX$ for $G>0$.

\end{proof}

\subsection{Asymptotic scheme}
In order to verify the consistency of the scheme at the diffusive limit,
we perform the limit $\eps\to 0$ in the above numerical scheme \eqref{upwdiff}--\eqref{eq:pdiff} when $G>0$ (the case $G<0$ being done in a similar way).
\begin{lemma}\label{lem:diff}
  Let the sequence $(p_{i,k}^{\pm,n+\frac 12})_{i,k,n}$ be computed thanks to scheme \eqref{upwdiff}.
  When $\eps\to 0$, we have $p_{i,k}^{\pm,n+\frac 12} \to \bar{p}_i^{\pm,n} \delta_{k=0}$, where $\displaystyle \bar{p}_i^{\pm,n}:=\sum_{k=-K}^K p_{i,k}^{\pm,n}$.
\end{lemma}
\begin{proof}
  Letting $\eps\to 0$ into equation \eqref{upwdiff1}, we deduce that $J_{i,k+\frac 12}^{\pm,n+\frac 12} = J_{i,k-\frac 12}^{\pm,n+\frac 12}$, and for $k=0$, we have by definition \eqref{upwdiff2} $J_{i,\frac 12}^{\pm,n+\frac 12} = 0 = J_{i,-\frac 12}^{\pm,n+\frac 12}$. 
  Then, for any $i\in\ZZ$, $k\in\{-K,\ldots,K\}$, $n\in\NN$, the limit satisfies $(y_k)^+ p_{i,k}^{\pm,n+\frac 12} = (y_{k+1})^- p_{i,k+1}^{\pm,n+\frac 12}$. We deduce that for $k\neq 0$, we have $p_{i,k}^{\pm,n+\frac 12}=0$.
  Thus, for any $i\in\ZZ$, $k\in\{-K,\ldots,K\}$, $n\in\NN$, there exists a quantity, denoted $\bar{p}_i^{\pm,n+\frac 12}$, such that $p_{i,k}^{+,n+\frac 12}=\bar{p}_i^{\pm,n+\frac 12} \delta_{k=0}$.
  Finally, thanks to the no-flux boundary conditions, we deduce that $\bar{p}_i^{\pm,n+\frac 12} = \sum_k p_{i,k}^{\pm,n+\frac 12} = \sum_k p_{i,k}^{\pm,n} =: \bar{p}_{i}^{\pm,n}.$
\end{proof}

We may now pass to the limit into equation \eqref{eq:pdiff}.
We have
\begin{lemma}
  Formally, when $\eps\to 0$ the sequence $(p_{i,k}^{\pm,n})_{i,k,n}$ computed thanks to equations \eqref{upwdiff}--\eqref{eq:pdiff} converges towards $\frac 12 \rho_i^n \delta_{k=0}$, where the sequence $(\rho_i^n)_{i,n}$ solves the numerical scheme
\begin{equation}\label{eq:findiff}
\rho_i^{n+1} = \rho_i^n + \frac{\Delta t G}{\Delta x} \left(\frac{1}{\bar{\Lambda}_{\frac 12}} \rho_{i-1}^n - \frac{1}{\bar{\Lambda}_{-\frac 12}} \rho_{i}^n + \frac{1}{\bar{\Lambda}_{-\frac 12}} \rho_{i+1}^n - \frac{1}{\bar{\Lambda}_{\frac 12}} \rho_{i}^n\right).
\end{equation}
This scheme is consistent with the Keller-Segel equation \eqref{eq:KS}.
\end{lemma}
\begin{proof}
We first observe that when $\eps\to 0$ into \eqref{eq:pdiff}, we have $p_{i,k}^{+,n+1} = p_{i,k}^{-,n+1}$.
Hence, for any $n\geq 1$, using also Lemma \ref{lem:diff}, we get that $\displaystyle \lim_{\eps \to 0} p_{i,k}^{+,n+\frac 12} = \lim_{\eps \to 0} p_{i,k}^{-,n+\frac 12} = \bar{p}_{i}^n \delta_{k=0}$.

Moreover, adding the two equations \eqref{eq:pplusdiff} and \eqref{eq:pmoinsdiff} after a change of index, we get
\begin{align*}
  & p_{i,k}^{+,n+1} + p_{i,k}^{-,n+1} =  p_{i,k}^{+,n+\frac 12} + p_{i,k}^{-,n+\frac 12} \\
  & \quad + \frac{2\DT G}{\DX}\left(\frac{1}{2\eps G + \bar{\Lambda}_{k-\frac 12}}(p_{i-1,k-1}^{+,n+\frac 12} - p_{i,k}^{-,n+\frac 12}) + \frac{1}{2\eps G + \bar{\Lambda}_{k+\frac 12}}(p_{i+1,k+1}^{-,n+\frac 12} - p_{i,k}^{+,n+\frac 12})\right).
\end{align*}
Passing into the limit $\eps\to 0$, we get, thanks to Lemma \ref{lem:diff},
\begin{align*}
  & (\bar{p}_{i}^{n+1} + \bar{p}_{i}^{n+1})\delta_{k=0} =  (\bar{p}_{i}^{n} + \bar{p}_{i,k}^{n})\delta_{k=0} \\
  & \qquad + \frac{2\DT G}{\DX}\left(\frac{1}{\bar{\Lambda}_{k-\frac 12}}(\bar{p}_{i-1}^{n}\delta_{k=1} - \bar{p}_{i}^{n} \delta_{k=0}) + \frac{1}{\bar{\Lambda}_{k+\frac 12}}(\bar{p}_{i+1}^{-,n}\delta_{k=-1} - \bar{p}_{i}^{+,n}\delta_{k=0})\right).
\end{align*}
We denote by $\displaystyle \rho_i^n := \sum_{k=-K}^K (p_{i,k}^{+,n} + p_{i,k}^{-,n}) = 2\bar{p}_i^n$.
By summing over $k$ the latter identity, we obtain \eqref{eq:findiff}. \\

We now verify that the limiting scheme \eqref{eq:findiff} is consistent with the Keller-Segel equation \eqref{eq:KS}.

Recalling that $\bar{\Lambda}_{\frac 12} = \int_0^{\Delta y} \Lambda(y)\,dy$, we approximate thanks to a trapezoidal rule $\bar{\Lambda}_{\frac 12} \simeq \frac{\Lambda(0)+\Lambda(\Delta y)}{2} \Delta y \simeq \Lambda(0) \Delta y + \frac 12 \Delta y^2 \Lambda'(0)$. Then,
\[
\frac{1}{\bar{\Lambda_{\frac 12}}} \simeq \frac{1}{\Lambda(0)\Delta y(1+\frac 12 \Delta y\frac{\Lambda'(0)}{\Lambda(0)})} \simeq \frac{1}{\Lambda(0)\Delta y} \left(1-\frac 12 \Delta y\frac{\Lambda'(0)}{\Lambda(0)}\right).
\]
By the same token we deduce,
\[
\frac{1}{\bar{\Lambda_{-\frac 12}}} \simeq \frac{1}{\Lambda(0)\Delta y} \left(1+\frac 12 \Delta y\frac{\Lambda'(0)}{\Lambda(0)}\right).
\]
Injecting this approximation into \eqref{eq:findiff}, recalling the relation $\Delta y = |G|\Delta x$, we deduce
\begin{align}\label{eq:fdks}
  \rho_i^{n+1} \simeq\ & \rho_i^n + \frac{\Delta t}{\Delta x^2 \Lambda(0)} \left(\rho_{i-1}^n - 2\rho_i^n + \rho_{i+1}^n \right) + \frac{\Delta t G \Lambda'(0)}{2\DX (\Lambda(0))^2} \left(\rho_{i+1}^n - \rho_{i-1}^n\right).
\end{align}
This is a finite difference centered discretization of the Keller-Segel equation \eqref{eq:KS}.
\end{proof}

\section{Hyperbolic scaling}\label{sec:hyperbolic}
\subsection{System at hyberbolic scaling}

In this part, we take $\sigma=\lambda_0=1, \tau=\eps$ into \eqref{twovelo}. The system at hyperbolic scaling reads
\begin{equation}
  \label{twovelo1}
  \pa_t p^\pm \pm  \pa_x p^\pm + \frac{1}{\eps} \pa_y((\pm G -y) p^\pm) = \pm \frac{\Lambda(y)}{2}(p^- - p^+).
\end{equation}
Formally, when $\eps\to 0$, we obtain $\pa_y((\pm G -y) p^\pm) = 0$. Then, we deduce that $p^{+}(t,x,y) \to f^+(t,x) \delta_{y= G}$ and $p^{+}(t,x,y) \to f^-(t,x) \delta_{y=- G}$ as $\eps\to 0$.
Then, integrating \eqref{twovelo1} over $y$, we deduce
\begin{equation}\label{eq:kin}
  \pa_t f^\pm \pm \pa_x f^\pm = \pm \frac{1}{2}\big(\Lambda(-G) f^- - \Lambda(G) f^+ \big).
\end{equation}
We recover the two stream kinetic system for chemotaxis.

In order to discretize equation \eqref{twovelo1}, we proceed in two steps :
 \begin{itemize}
 \item In a first step, we discretize the transport equation
   \[
   \pa_t p^\pm + \frac{1}{\eps} \pa_y((\pm G -y) p^\pm) = 0.
   \]
   As above we use an implicit upwind scheme, which reads, for any $n\in\NN$, $i\in\ZZ$ and $k\in\{-K,\ldots,K\}$,
   \begin{subequations}\label{upwhyper}
     \begin{align}\label{2v1}
       & p_{i,k}^{\pm,n+\frac 12}  = p_{i,k}^{\pm,n} - \frac{\DT}{\eps\DY}\Big(J_{i,k+\frac 12}^{\pm,n+\frac 12} - J_{i,k-\frac 12}^{\pm,n+\frac 12}\Big), \\[1mm]
       & J_{i,k+\frac 12}^{\pm,n+\frac 12} = (\pm G - y_k)^+ p_{i,k}^{\pm,n+\frac 12} - (\pm G - y_{k+1})^- p_{i,k+1}^{\pm,n+\frac 12}.\label{2v2}
     \end{align}
   \end{subequations}
   We impose no-flux boundary conditions~: $J_{i,K+\frac 12}^{\pm,n+\frac 12} = J_{i,-K-\frac 12}^{\pm,n+\frac 12} = 0$, for any $i\in\ZZ$.

 \item In the second step, we discretize
   \begin{equation}\label{eqppm}
     \pa_t p^\pm \pm \pa_x p^\pm = \pm \frac{\Lambda(y)}{2} (p^--p^+).
   \end{equation}
   Using a well-balanced scheme, in the spirit of \cite{GT, Numerische}, we get,
   for any $n\in\NN$, $i\in\ZZ$ and $k\in\{-K,\ldots,K\}$,
 \begin{subequations}\label{wbhyper}
   \begin{align}\label{2v3}
     p_{i,k}^{+,n+1}  =\ & p_{i,k}^{+,n+\frac 12} - \frac{\DT}{\DX}(p_{i,k}^{+,n+\frac 12}-p_{i,k}^{-,n+\frac 12}) 
      + \frac{\DT}{\DX(1+\frac{1}{2}\Lambda(y_k)\DX)}(p_{i-1,k}^{+,n+\frac 12}-p_{i,k}^{-,n+\frac 12})   \\
     p_{i-1,k}^{-,n+1}  =\ & p_{i-1,k}^{-,n+\frac 12} + \frac{\DT}{\DX}(p_{i-1,k}^{+,n+\frac 12}-p_{i-1,k}^{-,n+\frac 12}) 
     + \frac{\DT}{\DX(1+\frac{1}{2}\Lambda(y_k)\DX)}(p_{i,k}^{-,n+\frac 12}-p_{i-1,k}^{+,n+\frac 12}). \label{2v4} 
   \end{align}
 \end{subequations}
\end{itemize}

It is well-known that the numerical scheme \eqref{wbhyper} is stable and consistent with \eqref{eqppm} provided a CFL condition is satisfied:
\begin{lemma}
  Let us assume that the following CFL condition
  $
  \DT \leq \DX
  $
  holds and $\Lambda$ is a nonnegative function. Then the scheme \eqref{upwhyper}--\eqref{wbhyper} is positive and conservative.
\end{lemma}
\begin{proof}
  Indeed, scheme \eqref{upwhyper} is clearly positive and conservative.
  For \eqref{wbhyper}, it may be rewritten
  \begin{align*}
    p_{i,k}^{+,n+1}  =\ & p_{i,k}^{+,n+\frac 12} \left(1 - \frac{\DT}{\DX}\right) +  \frac{\DT}{\DX} \left(1-  \frac{1}{1+\frac{1}{2}\Lambda(y_k)\DX}\right) p_{i,k}^{-,n+\frac 12} 
                          + \frac{\DT}{\DX(1+\frac{1}{2}\Lambda(y_k)\DX)}p_{i-1,k}^{+,n+\frac 12}    \\
     p_{i-1,k}^{-,n+1}  =\ & p_{i-1,k}^{-,n+\frac 12} \left(1 - \frac{\DT}{\DX}\right) + \frac{\DT}{\DX}\left(1-\frac{1}{1+\frac{1}{2}\Lambda(y_k)\DX}\right) p_{i-1,k}^{-,n+\frac 12} 
     + \frac{\DT}{\DX(1+\frac{1}{2}\Lambda(y_k)\DX)} p_{i,k}^{-,n+\frac 12}.
  \end{align*}
  Then, since $\Lambda$ is nonnegative, it is clear that all coefficients are positive under the CFL condition.
\end{proof}

\subsection{Asymptotic limit}

 We verify the behaviour of the scheme when $\eps\to 0$.

 \begin{lemma}\label{lem:support}
   Let the sequence $(p_{i,k}^{\pm,n+\frac 12})_{i,k,n}$ be computed thanks to scheme \eqref{upwhyper}. When $\eps\to 0$, we have $p_{i,k}^{\pm,n+\frac 12} \to \bar{p}_{i}^{\pm,n} \delta_{k=\pm K}$, where $\displaystyle \bar{p}_i^{\pm,n} = \sum_{k=-K}^K p_{i,k}^{\pm,n}$.  \\
   Moreover, the sequence $(\bar{p}_i^{\pm,n})_{i,k,n}$ satisfies the numerical scheme
\begin{subequations}\label{eq:plimkin}
\begin{align}
   \bar{p}_{i}^{+,n+1} =\ & \bar{p}_{i}^{+,n} - \frac{\DT}{\DX}(\bar{p}_{i}^{+,n}-\bar{p}_{i}^{-,n}) + \frac{\DT}{\DX(1+\frac 12 \Lambda(G)\DX)}\bar{p}_{i-1}^{+,n}  
 - \frac{\DT}{\DX(1+\frac 12 \Lambda(-G)\DX)}\bar{p}_{i}^{-,n},  \label{eq:plimkina} 
\end{align}
\begin{align}
\bar{p}_{i-1}^{-,n+1} =\ & \bar{p}_{i-1}^{+,n} + \frac{\DT}{\DX}(\bar{p}_{i-1}^{+,n}-\bar{p}_{i-1}^{-,n}) + \frac{\DT}{\DX(1+ \frac 12 \Lambda(-G)\DX)} \bar{p}_{i}^{-,n} 
    - \frac{\DT}{\DX(1+ \frac 12 \Lambda(G)\DX)}\bar{p}_{i-1}^{+,n}.   \label{eq:plimkinb}
\end{align}
\end{subequations}
This is a consistent discretization of the kinetic equation \eqref{eq:kin}.
 \end{lemma}
 \begin{proof}
   Letting $\eps\to 0$ in \eqref{2v1}, we deduce that at the limit we have, for any $k\in\{-K,\ldots,K\}$, $J_{i,k+\frac 12}^{\pm,n+\frac 12} = J_{i,k-\frac 12}^{\pm,n+\frac 12} = 0$, where the latter equality is a consequence of the no-flux boundary condition.
     Then, for any $i,k,n$, the limit satisfies
     \[
     (\pm G - y_k)^+ p_{i,k}^{\pm,n+\frac 12} = (\pm G - y_{k+1})^- p_{i,k+1}^{\pm,n+\frac 12}.
     \]
     If $G>0$, we choose $k+1 = K$ and deduce that $0=p_{i,K-1}^{+,n+\frac 12}$. By induction $0 = p_{i,k}^{+,n+\frac 12}$ for any $k\leq K-1$.
     For $k=-K$, we deduce that $0 = p_{i,-K+1}^{-,n+\frac 12}$, by induction, $0 = p_{i,k}^{-,n+\frac 12}$ for any $k\geq -K$.

     If $G<0$, by taking $k+1 = K$, we get $0=p_{i,K-1}^{-,n+\frac 12}$. By induction, we deduce $0 = p_{i,k}^{-,n+\frac 12}$ for any $k\in\{-K,\ldots,K\}$.
     For $k=-K$, we get $0 = p_{i,-K+1}^{+,n+\frac 12}$. Hence, $0 = p_{i,k}^{+,n+\frac 12}$ for any $k\in\{-K,\ldots,K\}$.
     
     Moreover, summing over $k$ the equation \eqref{2v1}, we deduce that $\sum_{k=-K}^K p_{i,k}^{\pm,n+\frac 12} = \sum_{k=-K}^K p_{i,k}^{\pm,n}$.

 Hence, summing \eqref{2v3} and \eqref{2v4} over $k$, we obtain the scheme \eqref{eq:plimkin}. 
 In order to verify the consistency with the kinetic equation \eqref{eq:kin}, we recall the expansion
\[
\frac{1}{1+ \frac 12 \Lambda(\pm G)\DX} = 1 - \frac 12 \Lambda(\pm G)\DX + o(\DX).
\]
Injecting this expression into \eqref{eq:plimkina}, we obtain 
\begin{align*}
  \bar{p}_i^{+,n+1} =\  & \bar{p}_i^{+,n} - \frac{\DT}{\DX}\left(\bar{p}_i^{+,n} - \bar{p}_{i-1}^{+,n} + \frac 12 \Lambda(G) \DX \bar{p}_{i-1}^{+,n} - \frac 12 \Lambda(-G) \DX \bar{p}_{i}^{-,n} + o(\DX)\right)  \\
=\ & \bar{p}_i^{+,n} - \frac{\DT}{\DX} (\bar{p}_i^{+,n} - \bar{p}_{i-1}^{+,n}) - \frac{\DT}{2} \big(\Lambda(G) \bar{p}_{i-1}^{+,n} -\Lambda(-G) \bar{p}_{i}^{-,n}\big) + o(\DT).
\end{align*}
This is consistent with the equation for $f^+$ in \eqref{eq:kin}.
We proceed in the same way with \eqref{eq:plimkinb}.
\end{proof}

\subsection{A remark on the extension of this latter scheme to the diffusive regime}
\label{subsec:wrong}

In this subsection, we underline the importance of a careful use of the splitting approach to recover the good asymptotic limit. Indeed, a natural extension of the scheme used in previous section to the diffusive regime will not provide a consistent scheme at the diffusive limit.
In order to justify this affirmation, let us consider the following numerical scheme:
 \begin{itemize}
 \item We use the same upwind scheme as above for a first time step \eqref{2v1}--\eqref{2v2}.
   
 \item In the second step, we discretize by a uniform scheme the equation
   \[
   \eps^2\pa_t p^\pm \pm \eps \pa_x p^\pm = \pm \frac{\Lambda(y)}{2} (p^--p^+).
   \]
   Using an asymptotic preserving well-balanced scheme, in the spirit of \cite{GT, Numerische}, we get
   \begin{align}\label{2v5}
     p_{i,k}^{+,n+1} =\ & p_{i,k}^{+,n+\frac 12} - \frac{\DT}{\eps\DX}(p_{i,k}^{+,n+1}-p_{i,k}^{-,n+1}) \\
     & + \frac{\DT}{\DX(\eps+\frac{1}{2}\Lambda(y_k)\DX)}(p_{i-1,k}^{+,n+\frac 12}-p_{i,k}^{-,n+\frac 12})  \nonumber \\
     p_{i-1,k}^{-,n+1} =\ &  p_{i-1,k}^{-,n+\frac 12} + \frac{\DT}{\eps\DX}(p_{i-1,k}^{+,n+1}-p_{i-1,k}^{-,n+1}) \label{2v6}  \\
     & + \frac{\DT}{\DX(\eps+\frac{1}{2}\Lambda(y_k)\DX)}(p_{i,k}^{-,n+\frac 12}-p_{i-1,k}^{+,n+\frac 12}). \nonumber
   \end{align}

   Since this scheme is uniformly accurate with respect to $\eps$, we may pass to the limit $\eps\to 0$. 
   After straightforward computations, we get from \eqref{2v5}--\eqref{2v6}
   \begin{align*}
     & p_{i,k}^{+,n+1} = \frac{\eps \DX + \DT}{\eps \DX + 2\DT} A_{i-\frac 12,k}^{n+\frac 12} + \frac{\DT}{\eps \DX + 2\DT} B_{i+\frac 12,k}^{n+\frac 12} \\[1mm]
     & p_{i,k}^{-,n+1} = \frac{\DT}{\eps \DX + 2\DT} A_{i-\frac 12,k}^{n+\frac 12} + \frac{\eps \DX + \DT}{\eps \DX + 2\DT} B_{i+\frac 12,k}^{n+\frac 12} \\[1mm]
     & A_{i-\frac 12,k}^{n+\frac 12} = p_{i,k}^{+,n+\frac 12} + \frac{\DT}{\DX(\eps + \frac 12 \Lambda(y_k)\DX)} (p_{i-1,k}^{+,n+\frac 12} - p_{i,k}^{-,n+\frac 12}) \\[1mm]
     & B_{i+\frac 12,k}^{n+\frac 12} = p_{i,k}^{-,n+\frac 12} + \frac{\DT}{\DX(\eps + \frac 12 \Lambda(y_k)\DX)} (p_{i+1,k}^{-,n+\frac 12} - p_{i,k}^{+,n+\frac 12}).
   \end{align*}
   As a first observation, we notice that when $\eps\to 0$, we have
   \[
   \lim_{\eps\to 0} p_{i,k}^{+,n+1} = \lim_{\eps\to 0} \frac 12(A_{i-\frac 12,k}^{n+\frac 12}+B_{i+\frac 12,k}^{n+\frac 12}) =
   \lim_{\eps\to 0} p_{i,k}^{-,n+1}.
   \]
   We deduce that for any $n\in\NN^*$, we have $\lim_{\eps\to 0} p_{i,k}^{+,n} = \lim_{\eps\to 0} p_{i,k}^{-,n}$.
   
   Moreover, by adding the first two equalities in the above numerical scheme, we get
   \[
   p_{i,k}^{+,n+1} + p_{i,k}^{-,n+1} = A_{i-\frac 12,k}^{n+\frac 12} + B_{i+\frac 12,k}^{n+\frac 12}
   \]
   Let us define the macroscopic density
   $\rho_i^n := \sum_{k=-K}^K (p_{i,k}^{+,n} + p_{i,k}^{-,n})$. Then, we have clearly
   \[
   \rho_i^{n+1} = \sum_{k=-K}^K \Big(A_{i-\frac 12,k}^{n+\frac 12} + B_{i+\frac 12,k}^{n+\frac 12}\Big).
   \]
   We pass to the limit $\eps\to 0$ in the latter numerical scheme. We recall that from Lemma \ref{lem:support} (ii), we have $p_{i,k}^{\pm,n+\frac 12} \to \bar{p}_{i}^{\pm,n} \delta_{k=\pm K}$ where $\bar{p}_i^{\pm,n} = \sum_{k=-K}^K p_{i,k}^{\pm,n}$.
   Hence, passing into the limit, we obtain
   \begin{align*}
     & A_{i-\frac 12,k}^{n+\frac 12} \to  \Big(\bar{p}_{i}^{+,n} + \frac{2\DT}{\Lambda(G)\DX^2} \bar{p}_{i-1}^{+,n}\Big) \delta_{k=K} - \frac{2\DT}{\Lambda(-G)\DX^2} \bar{p}_{i}^{-,n} \delta_{k=-K}  \\
     & B_{i+\frac 12,k}^{n+\frac 12} \to  \Big(\bar{p}_{i}^{-,n} + \frac{2\DT}{\Lambda(-G)\DX^2} \bar{p}_{i+1}^{-,n}\Big) \delta_{k=-K} - \frac{2\DT}{\Lambda(G)\DX^2} \bar{p}_{i}^{+,n} \delta_{k=K}.
   \end{align*}
   Summing over $k$, and denoting $\bar{\rho}_i^n = \lim_{\eps\to 0} \rho_i^n$, we arrive to the limiting scheme
   \[
   \bar{\rho}_i^{n+1} = \bar{\rho}_i^{n} + \frac{2\DT}{\DX^2}\left(\frac{1}{\Lambda(G)}(\bar{p}^{+,n}_{i-1} - \bar{p}_i^{+,n}) + \frac{1}{\Lambda(-G)}(\bar{p}_{i+1}^{-,n}-\bar{p}_{i}^{-,n})\right).
   \]
   With the fact that $\bar{p}^{+,n}_{i} = \bar{p}^{-,n}_{i} = \frac 12 \bar{\rho}_i^n$, we get
   \[
   \bar{\rho}_i^{n+1} = \bar{\rho}_i^{n} + \frac{\DT}{\DX^2}\left(\frac{1}{\Lambda(G)}(\bar{\rho}^{n}_{i-1} - \bar{\rho}_i^{n}) + \frac{1}{\Lambda(-G)}(\bar{\rho}_{i+1}^{n}-\bar{\rho}_{i}^{n})\right).
   \]
   This latter scheme is obviously not consistent with the Keller-Segel equation \eqref{eq:KS}.

 \end{itemize}

\section{Numerical results}\label{Sec_Num}

\subsection{Problem}\label{problem}

In order to illustrate our theoretical results, we carry out numerical computations for the following problem.
Let us consider the one dimensional space $x\in[0,1]$ with mirror boundary conditions,
\begin{equation*}
p^+(t,x=0,y)=p^-(t,x=0,y),\qquad p^-(t,x=1,y)=p^+(t,x=1,y).
\end{equation*}
These boundary conditions impose 
that the net mass flux at the boundary is zero; i.e.,
\[
J=\int_\mathbb{R} (p^+-p^-)dy=0, \qquad \mathrm{at}\quad x=0,1.
\]
As the initial condition, we consider
\begin{equation*}
p^\pm(t=0,x,y)=\left\{
\begin{array}{ll}
\frac 1 G, \quad & |y|\le \frac{|G|}{2},\\[2mm]
0, & |y|>\frac{|G|}{2},
\end{array}
\right.
\end{equation*}
for any $x\in[0,1]$.
Hereafter we only consider the case where $G> 0$.

For this problem, the steady state solutions in the diffusive and hyperbolic limits are explicitly calculated as below~:

\begin{description}
\item[Steady state in the diffusive limit.]
  
Let us consider the Keller-Segel system \eqref{eq:KS}, which is obtained in the continuum limit at the diffusive scaling.
We first recall that this equation is conservative, implying that $\int_0^L \bar{p}_0(t,x)\,dx = M$ for all positive time $t$.
The steady states verify
  \[
  \pa_x \left(\frac{1}{\Lambda(0)} \pa_x \bar{p}_0 + G\bar{p}_0 \frac{\Lambda'(0)}{\Lambda(0)^2}\right) = 0.
  \]
 Since the flux in the diffusive scaling is given by (\ref{eq:p1}), the no flux boundary conditions give
  \[
  \pa_x \bar{p}_0 + G\bar{p}_0 \frac{\Lambda'(0)}{\Lambda(0)} = 0.
  \]
  By integrating the above equation and using the mass conservation, we obtain
  \begin{equation}\label{ks_steady}
  \bar{p}_0(x) = \frac{G\Lambda'(0) M}{\Lambda(0)\Big(1-e^{-\frac{G\Lambda'(0)}{\Lambda(0)}L}\Big)} e^{-\frac{G\Lambda'(0)}{\Lambda(0)}x}.
  \end{equation}

\item[Steady state in the hyperbolic limit.]
  
  For the two-stream kinetic equation in the hyperbolic limit (\ref{eq:kin}), the steady state can be explicitly computed.
  Indeed, it should satisfy~:
  \begin{equation}\label{eq:kinlim}
    \pa_x f^\pm = \frac 12 (\Lambda(-G) f^- - \Lambda(G) f^+),
  \end{equation}
  complemented with no-flux boundary conditions.
  By subtracting the above equations, we obtain
  $\pa_x (f^+-f^-) = 0$. Thus $f^+ - f^-$ is constant, which is $0$ thanks to the no-flux boundary conditions. Therefore, $f^+ = f^- = f$, injecting into \eqref{eq:kinlim}, it gives
  \[
  \pa_x f = \frac 12 (\Lambda(-G) -\Lambda(G)) f.
  \]
  Integrating,
  \[
  f(x) = A e^{ \frac 12 (\Lambda(-G) -\Lambda(G)) x}.
  \]
  The constant $A$ is computed thanks to the conservation of the mass~: $M=\int_0^L f(x)\,dx$. Finally, we have
  \begin{equation}\label{analytic_hyplimit}
  f(x) = \frac{(\Lambda(-G)-\Lambda(G))M}{2(e^{\frac 12(\Lambda(-G)-\Lambda(G))L}-1)} e^{\frac 12(\Lambda(-G)-\Lambda(G))x}.
  \end{equation}

\end{description}

\subsection{Numerical results for scheme (\ref{upwdiff})--(\ref{eq:pdiff})}
\label{sec_numdiff}

%
The asymptotic preserving (AP-diff) scheme (\ref{upwdiff})--(\ref{eq:pdiff}) is implemented for the problem in Section \ref{problem}.
Although this scheme is implicit, it can be implemented very efficiently.
The computational procedure is described in Appendix \ref{app_ap}.
We carry out the numerical computations for various values of $\lambda_0$ and illustrate the validity of the AP-diff scheme by comparing the numerical results to those obtained by a MC method, which is explained in Appendix \ref{app_mc}, and the analytical solution (\ref{ks_steady}).
We also investigate the accuracy of the AP-diff scheme by comparing the numerical results for different mesh systems.

For the response function $\Lambda(y)$, we consider
\begin{equation}\label{eq_Lambda}
\Lambda(y)=1-\chi\arctan(y),
\end{equation}
where $\chi$ is the modulation amplitude.
In the following computations, unless otherwise stated, the parameters $G=1$ and $\chi=0.5$ are fixed.

\begin{figure}[htbp]
\centering
\includegraphics[width=1.0\textwidth]{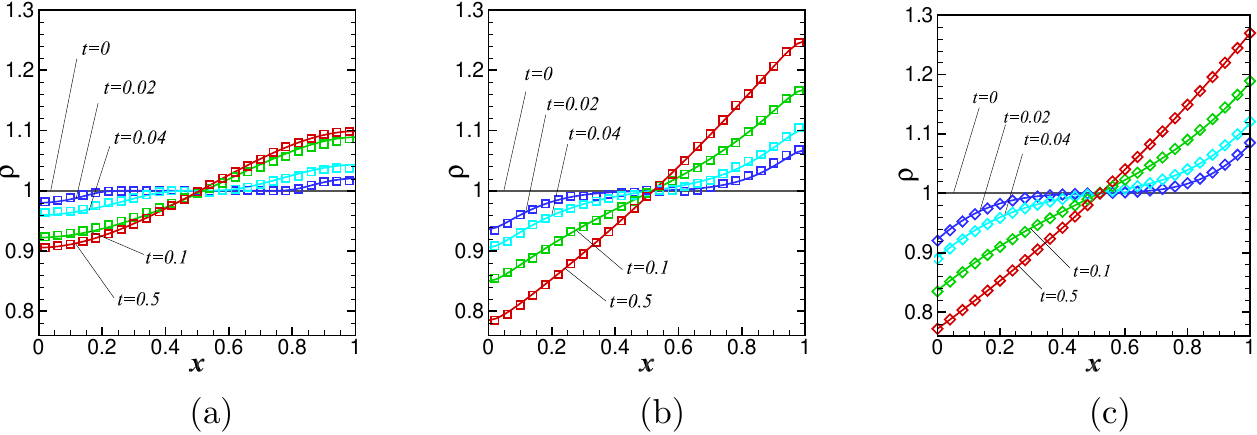}
\caption{
Time dynamics of $\rho$ for different values of $\lambda_0$; i.e., $\lambda_0=10$ in \rm{(a)}, $\lambda_0=10^3$ in \rm{(b)}, and $\lambda_0=10^8$ in \rm{(c)}.
The solid lines show the results obtained by the AP-diff scheme (\ref{upwdiff})-- (\ref{eq:pdiff}).
The squared symbols show the results obtained by the MC method in \rm{(a)} and \rm{(b)} and those obtained by the finite-difference scheme of KS equation (\ref{eq:fdks}) in \rm{(c)}.
}\label{fig:tevol}
\end{figure}

Figure \ref{fig:tevol} shows the time dynamics of population density $\rho$ for different values of $\lambda_0$.
In Figure \ref{fig:tevol}(a) and \ref{fig:tevol}(b), the results of the AP-diff scheme are compared with those obtained by the MC method.

Although the AP-diff scheme and MC method are different types of numerical methods, both methods provide consistent numerical results.
The distributions in the internal state $y$ obtained by the AP-diff and MC methods are also compared in Figure \ref{fig:dist}.
It is seen that even for the distribution functions both results coincide with each other.
We also observe that the $y$-profile becomes narrower and more symmetric as $\lambda_0$ increases.
This observation is consistent with the asymptotic behavior for small $\eps=\lambda_0^{-1}$, where we know that $p^\pm$ concentrate at $y=0$ in the continuum limit, i.e. $p^\pm(t,x,y) \rightarrow p_0(t,x)\delta_{y=0}$ as $\eps\rightarrow 0$ (see Section \ref{sec:diffscal}).
\begin{figure}[htbp]
\centering
\includegraphics[width=0.5\textwidth]{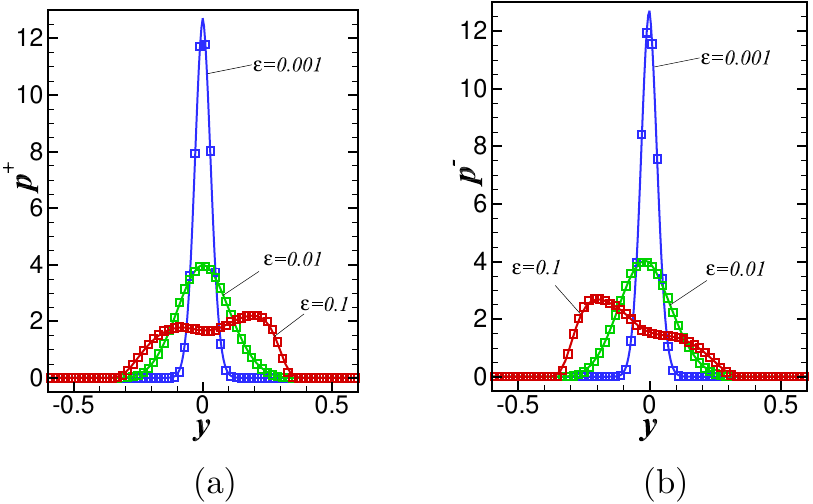}
\caption{
Figures \rm{(a)} and \rm{(b)} show $p^+(y)$ and $p^-(y)$, respectively, for different values of $\eps=\lambda_0^{-1}$.
The solid lines show the results obtained by the AP-diff scheme while the plots with squared symbols show the results obtained by the MC method.
}\label{fig:dist}
\end{figure}

However, when $\eps$ is too small, MC method may not be implemented efficiently.
Then in Figure \ref{fig:tevol}(c), the result of the AP-diff scheme for $\lambda_0=10^8$ is compared with the KS equation (\ref{eq:fdks}), which is the continuum limit equation when $\eps\to 0$ of the two stream kinetic model at diffusive scaling.
Remarkably, this comparison illustrates the asymptotic preserving nature of the AP-diff scheme.

Furthermore, Table \ref{table:acc_diff} shows the numerical accuracy of the AP-diff scheme for different values of $\lambda_0$.
It is seen that the AP-diff scheme is uniformly accurate and efficient irrespective of the parameter value of $\lambda_0$.
Even for the small mesh system $I=50$, the maximum relative error of the population density $\rho$, i.e., the error in the norm $L^\infty$ divided by the local population density, is estimated less than $0.1 \%$.
\begin{table}[htbp]
\caption{
Numerical accuracy of the AP-diff scheme for different values of $\lambda_0$.
The maximum relative errors, i.e. the error in the norm $L^\infty$ divided by the local population density, between the steady-state profiles of $\rho$ obtained for two different mesh systems ($I$,$I'$) are shown.
The mesh size is set as $\Delta x$=$\Delta y$=$1/I$ and the time-step size is set as $\Delta t=0.1\Delta x^2$.
Read as 4.9e-4=4.9$\times 10^{-4}$
}\label{table:acc_diff}
\centering
\begin{tabular}{c cccc}
\hline\hline
	Mesh ($I$,$I'$)& $\lambda_0=10$&$\lambda_0=10^2$&$\lambda_0=10^3$&$\lambda_0=10^{8}$\\
	\hline
	(50,200)& 4.7e-4  & 6.7e-4 &5.5e-4 &4.1e-5 \\
	(100,200)& 1.7e-4 & 2.4e-4 &1.9e-4 &8.1e-6 \\
	\hline\hline
\end{tabular}
\end{table}

\subsection{Numerical results for the second scheme (\ref{upwhyper})--(\ref{wbhyper})}

In this Section, we consider the scheme (\ref{upwhyper})--(\ref{wbhyper}), denoted AP-hyp scheme in the following, for the problem in Section \ref{problem} at the hyperbolic scaling \eqref{twovelo1}.
We carry out numerical computations for various values of $\tau$ and $\lambda_0$ and compare the results with those obtained by the MC method and the analytical solution in the limit $\tau\rightarrow 0$, (\ref{analytic_hyplimit}).

Figure \ref{fig:tevol_hyp} displays the time dynamics of population density $\rho$ for different values of $\tau$ when $\lambda_0=10$ is fixed.
In each figure, the results obtained by the AP-hyp scheme are compared with those obtained by the MC method.
It is clearly seen that the both methods can provide consistent results.
Furthermore, in figure (c), the steady-state profiles obtained by the AP-hyp and MC schemes are compared with the analytical solution in the asymptotic limit $\tau\to 0$ \eqref{analytic_hyplimit}.
This comparison illustrate the validity of the AP-hyp scheme when $\tau\to 0$.

The distributions in the internal state $y$ obtained by AP-hyp and MC methods are compared in Figure \ref{fig:dist_hyp}.
It is seen that both methods can provide consistent results.
It is also seen that the profile is completely different from that obtained at the diffusive scaling (see Fig. \ref{fig:dist}).
As described in Section \ref{sec:hyperbolic}, in the hyperbolic scaling, the distribution functions $p^\pm$ concentrates at $y=\pm G$, respectively, when $\tau\rightarrow 0$.
Figure \ref{fig:dist_hyp} illustrates that the AP-hyp scheme reproduces this asymptotic behavior.

\begin{figure}[htbp]
\centering
\includegraphics[width=1.0\textwidth]{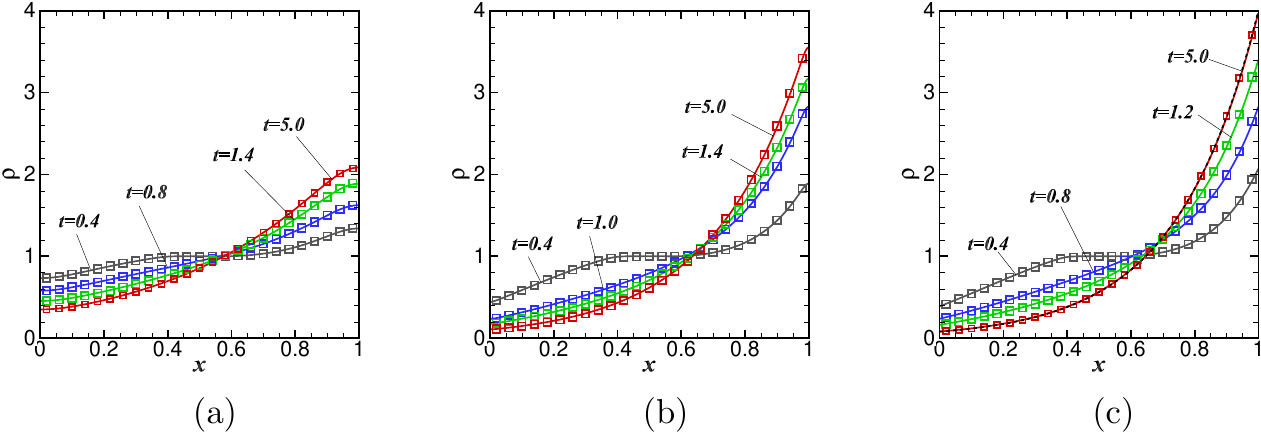}	
\caption{
Time dynamics of $\rho$ for different adaptation time $\tau$; i.e., (a) $\tau$=0.1, (b) $\tau$=0.01 , and (c) $\tau=10^{-8}$.
The parameter $\lambda_0=10.0$ is fixed. 
The solid lines show the results obtained by the AP-hyp scheme while the squared symbols show those obtained by the MC method.
In figure (c), the dashed line shows the analytical result in the limit $\tau\rightarrow 0$ calculated by Eq. (\ref{analytic_hyplimit}).
}\label{fig:tevol_hyp}
\end{figure}
\begin{figure}[htbp]
\centering
\includegraphics[width=0.5\textwidth]{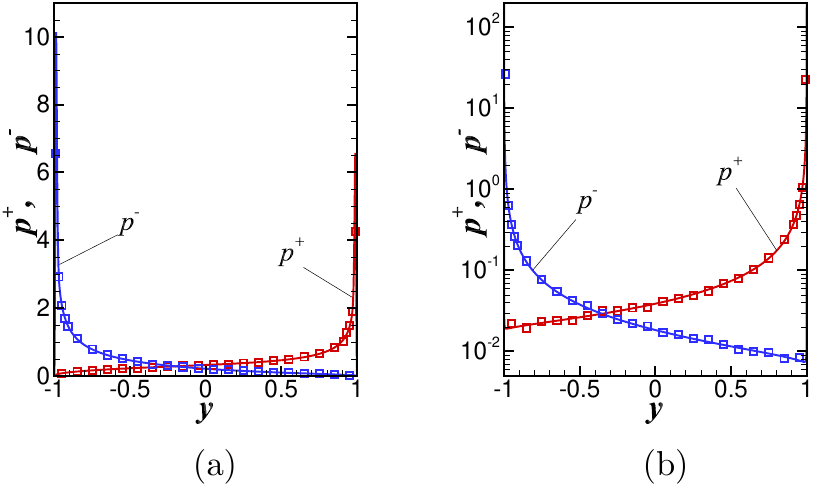}
\caption{
Figures \rm{(a)} and \rm{(b)} show the results for $\tau$=0.1 and $\tau$=0.01, respectively.
The parameters $\lambda_0=10$ is fixed.
The solid lines show the results of the AP-hyp scheme while the symbols show the results of the MC method.
}\label{fig:dist_hyp}
\end{figure}

Table \ref{table:acc_hyp} shows the numerical accuracy of the AP-hyp scheme.
Compared with Table \ref{table:acc_diff}, it is seen that the AP-hyp scheme is less accurate than the AP-diff scheme for $\tau=1$.
However, the AP-hyp scheme keeps the accuracy even for a very small $\tau$ at $\lambda_0=10$.

When $\lambda_0$ is large, the accuracy of the AP-hyp scheme significantly falls at small values of $\tau$.
This is because the spatial profile of $\rho$ becomes exponential for small values of $\tau$ when $\lambda_0$ is large.
Indeed, the analytical solution for $\tau\rightarrow 0$ (\ref{analytic_hyplimit}) indicates that the mass is concentrated at $x=1$ when $\lambda_0$ is large.
Figure \ref{fig:compari_tau} illustrates that the spatial profile of $\rho$ becomes exponential when $\tau$ is small at $\lambda_0=100$.
The growth rate is inversely proportional to $\tau$.

\begin{table}[htbp]
\caption{
Numerical accuracy of the AP-hyp scheme (\ref{upwhyper}) and (\ref{wbhyper}) for different values of $\tau$ at $\lambda_0$=10 and 100.
The maximum relative errors, i.e. the error in the norm $L^\infty$ divided by the local population density, between the spatial profiles of $\rho$ obtained in two different mesh systems ($I$,$I'$) are shown.
The mesh size $\Delta y$ is set as $\Delta y=\Delta x$.
See also the caption in Table \ref{table:acc_diff}.
}\label{table:acc_hyp}
\centering
\begin{tabular}{c cccc cccc}
\hline\hline
Mesh &\multicolumn{4}{c}{$\lambda_0=10$}&&
\multicolumn{3}{c}{$\lambda_0=100$}\\
\cline{2-5}\cline{7-9}
	($I$,$I'$)&
	$\tau$=1.0 &$\tau$=0.1 & $\tau$=0.01 &$\tau$=$10^{-8}$ 
	&&$\tau=1.0$ &$\tau=0.1$ &$\tau$=0.02\\
	\hline
	(200,800) & 3.1e-3 & 2.0e-2& 5.9e-2 & 5.6e-2 
	         && 1.0e-2 & 1.0e-1& --\\
	(400,800) & 1.0e-3 & 6.7e-3& 1.9e-2 & 1.9e-2 
	         && 3.5e-3 & 3.4e-2& 3.3e-1\\
	(800,1600)& --      & --     & --      & --
	         && --      & --     & 1.6e-1\\
	\hline\hline
\end{tabular}
\end{table}
\begin{figure}[htbp]
\centering
\includegraphics[width=0.25\textwidth]{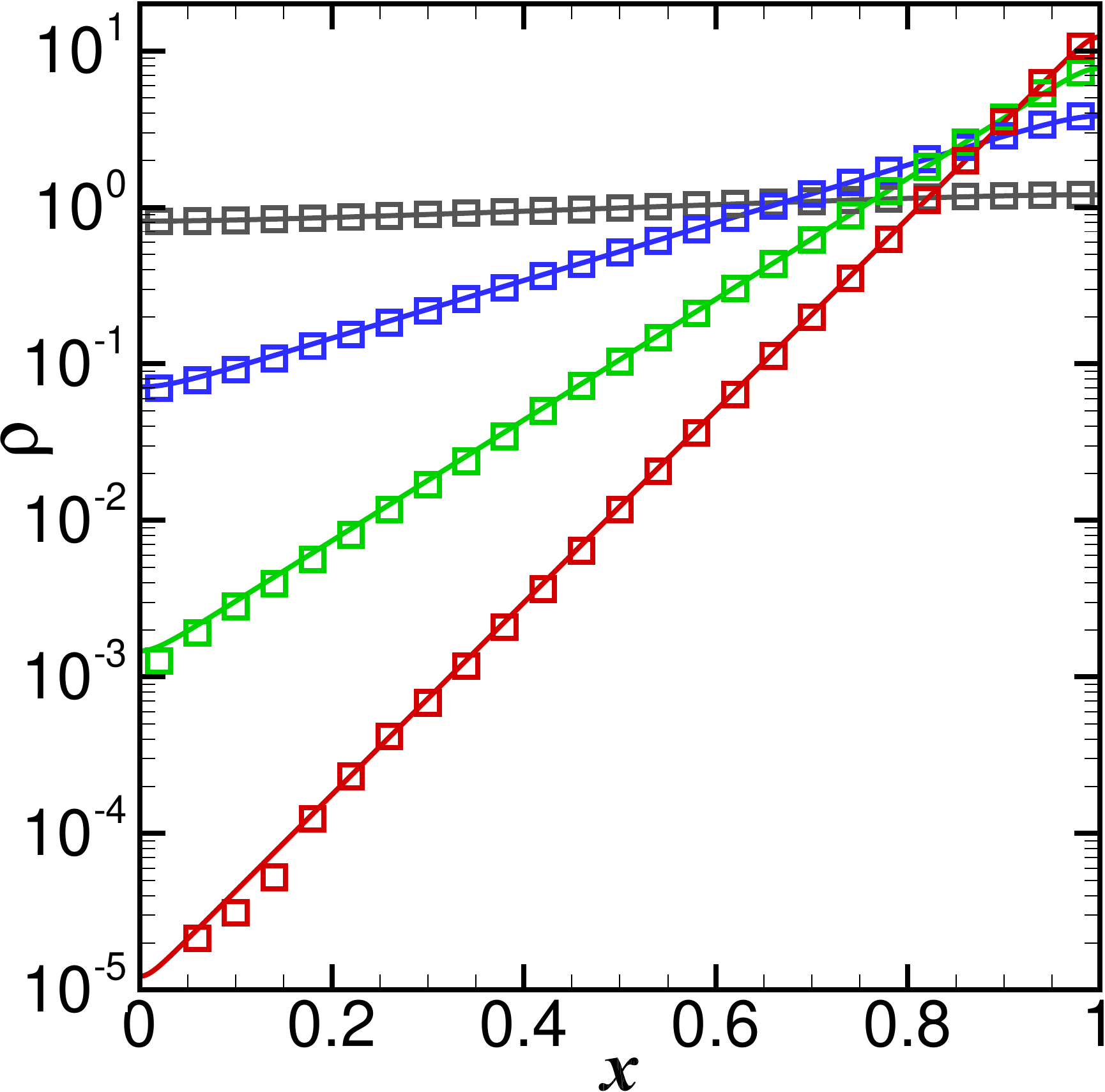}
\caption{
Spatial profiles of population density $\rho$ for different values of $\tau$.
The parameter $\lambda_0=100$ is fixed.
The solid lines show the results of the AP-hyp scheme and the symbols show the results of the MC method.
}\label{fig:compari_tau}
\end{figure}

\subsection{Comparison of two schemes}

\begin{table}[htbp]
\caption{
Numerical accuracy of the AP-diff scheme for different values of $\tau$ at $\lambda_0$=10 and 100.
The maximum relative errors between the spatial profiles of $\rho$ obtained in two different mesh systems ($I$,$I'$) are shown.
The mesh size $\Delta y$ is set as $\Delta y=\frac{|G|}{\tau}\Delta x$.
See also the caption in Table \ref{table:acc_diff}.
}\label{table:acc_ap2}
\centering
\begin{tabular}{c cc c cc}
\hline\hline
Mesh &\multicolumn{2}{c}{$\lambda_0=10$}&&
\multicolumn{2}{c}{$\lambda_0=100$}\\
\cline{2-3}\cline{5-6}
	($I$,$I'$)&
	$\tau=0.1$ &$\tau=0.01$
	&&$\tau=0.1$ &$\tau=0.02$\\
	\hline
	(200,800) & 8.9e-3 & -- 
			 && 2.8e-2 & --\\
	(400,800) & 3.0e-3 & 6.8e-2 
	         && 9.4e-3 & 1.7e-1 \\
	(800,1600)& --      & 3.5e-3   
	         && --      & 8.6e-2   \\ 
	\hline\hline
\end{tabular}
\end{table}
\begin{figure}
\centering
\includegraphics[width=0.5\textwidth]{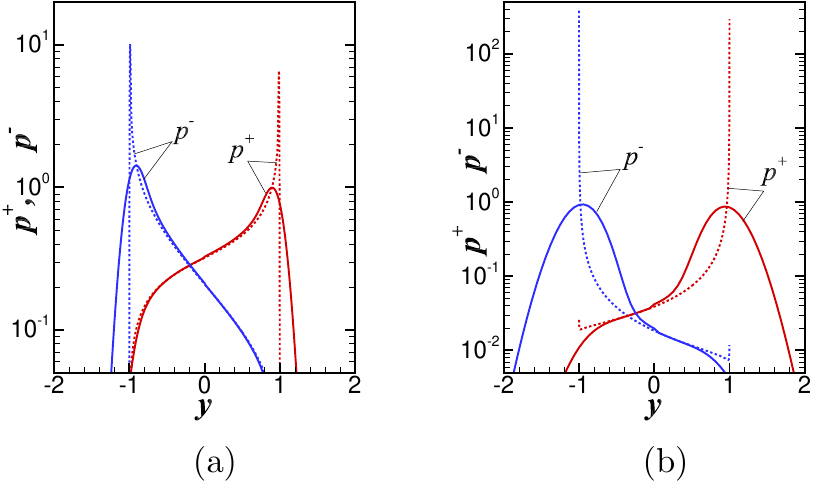}
\caption{
Comparison of the $y$ distribution obtained by the AP-diff scheme (solid lines) and that obtained by the hyperbolic scheme (dashed lines).
Figure (a) and (b) show the results for $\tau=0.1$ and $\tau=0.01$, respectively, at $\lambda_0=10$.
}\label{fig:dist_hyp_diff}
\end{figure}

In Section \ref{sec_numdiff} we have implemented the AP-diff scheme for the two-stream kinetic system at diffusive scaling, where the parameter $\tau=1$ is fixed.
In this Section, in order to treat the variation of parameter $\tau$ explicitely, we slightly modify the AP-diff scheme as following.
We replace $\Delta y$ with $\tau \Delta y$ in (\ref{upwdiff}) and $G$ with $G/\tau$ in (\ref{eq:pdiff}) and (\ref{eq:pdiff2}).
The relation between $x$- and $y$- mesh intervals is rewritten as $\Delta y=\frac{|G|}{\tau}\Delta x$ while the CFL condition (\ref{CFL}) remains unchanged.
Thus, when the parameter $\tau$ is small, the mesh size in $y$ space, $\Delta y$ becomes more coarsened than that in $x$ space.

Although the diffusive scaling is not appropriate unless the parameter $\tau$ is sufficiently larger than $\lambda_0^{-1}$, we implement the modified AP-diff scheme for some small values of $\tau$ and investigate the applicability of the AP-diff scheme for small $\tau$ by comparing the results with those obtained by the AP-hyp scheme (\ref{upwhyper}) and (\ref{wbhyper}).

It should be noted that the AP-diff scheme is conservative when the condition $p^\pm_{i,K}=p^\pm_{i,-K}=0$ for any $i\in\{0,\ldots,I\}$ is satisfied (see Appendix \ref{mass_cons}).
Indeed, this condition is satisfied at the diffusive scaling since the internal state $y$ is concentrated at $y=0$ (See Fig. \ref{fig:dist}).
However, as it is observed in Fig. \ref{fig:dist_hyp}, the internal state $y$ is biased at $y=\pm G$ when $\tau$ is small.
Thus, when we implement the modified AP-diff scheme for a small value of $\tau$, we need to extend the domain of $y$ in order to satisfy the mass conservation.
In the following computations, we set the domain of $y$ as $-3G\le y\le 3G$ for $\tau=0.02$ and  $-4G\le y\le 4G$ for $\tau=0.01$ and divide the domain into the uniform mesh system with $\Delta y=\frac{|G|}{\tau}\Delta x$.

Table \ref{table:acc_ap2} shows the accuracy of the modified AP-diff method when the parameter $\tau$ varies.
Compared to Table \ref{table:acc_hyp}, it is found that the AP-diff scheme is more accurate than the AP-hyp scheme except for the case with $\tau=0.01$ and $\lambda_0=10$.
We also confirmed that the spatial profiles of $\rho$ obtained by the AP-diff scheme and by the AP-hyp scheme coincide with each other within the numerical accuracy.
Thus, these results confirm that the AP-diff scheme can accurately reproduce the macroscopic density $\rho$ even for the case where the parameter $\tau$ is as small as $\lambda_0^{-1}$.

However, very interestingly, the distribution of the internal state $y$ obtained by the AP-diff scheme is quite different from those obtained by the AP-hyp and MC methods.
This illustrate the remark in Subsection \ref{subsec:wrong}.
Figure \ref{fig:dist_hyp_diff} shows the distribution of the internal state $y$ at $x=0.5$ obtained by the AP scheme.
The distribution of $y$ obtained by the AP-hyp and MC methods for the same parameter values is shown in Fig. \ref{fig:dist}.
In Figs. \ref{fig:dist_hyp_diff} (a) and (b), the results obtained by the AP-hyp scheme are also shown for comparison.
The distributions of $y$ obtained by the AP scheme also have peaks at $y=\pm G$.
However, the profiles are much more diffusive than those obtained by the AP-hyp scheme.

\section{Conclusion}
\label{sec:conclusion}

This paper deals with the numerical discretization of a two-stream kinetic model for bacterial chemotaxis with internal state.
We have proposed two schemes depending on the scaling parameters~: AP-diff for the model at diffusive scaling, AP-hyp at hyperbolic scaling. Theses scheme have the property to be consistent with the asymptotic limit, which is the so-called asymptotic preserving property.
In numerical parts, we implemented both AP-diff and AP-hyp schemes.
Numerical convergence of the schemes is evaluated in detail by comparing the results obtained for different meshes.
We also compare the numerical results obtained by the AP-diff and AP-hyp schemes to those obtained by a MC method and some analytical results in order to confirm the validity and consistency of the proposed schemes.

These numerical results confirm that the AP-diff scheme is uniformly accurate and efficient with respect to the variation of the tumbling frequency $\lambda_0$ at the diffusive scaling.
From the comparisons between the AP-diff scheme and MC method, the consistency between the two methods is clearly observed.
Remarkably, the asymptotic preserving nature of the AP-diff scheme is clearly illustrated by comparing the result obtained at a very small $\lambda_0^{-1}$, i.e., $\lambda_0^{-1}=10^{-8}$, to the numerical solution of the KS equation.

The AP-hyp scheme is robust with respect to the variation of the parameter value $\tau$.
Even when the parameter $\tau$ is very small, i.e., $\tau=10^{-8}$, the AP-hyp scheme can keep the numerical accuracy.
The consistency of the AP-hyp scheme and MC method is also confirmed in the time dynamics of the macroscopic density and the distribution function of the internal sate at steady state.

Finally, we also compare the AP-diff scheme and the AP-hyp scheme after a slight modification of the original AP-diff scheme to treat the variation in $\tau$.
Although the AP-diff scheme loses the asymptotic preserving property when the parameter $\tau$ is comparable to or smaller than $\lambda_0^{-1}$, it accurately reproduces the macroscopic density for moderately small values of $\tau$, where $\tau$ is still larger than $\lambda_0^{-1}$.
However, interestingly, the distribution function of the internal state obtained by the AP-diff scheme is quite different from that obtained by the AP-hyp scheme and MC method for moderately small values of $\tau$.
This interesting observation illustrates the mismatch of the time splitting strategy and the asymptotic behavior of the solution.
The importance of a careful strategy in the time splitting approach is discussed in detail in \S\ref{subsec:wrong}.

In summary, we conclude that the AP-diff scheme is very accurate and efficient at the diffusive scaling even in the continuum limit.
Even when the parameter $\tau$ is moderately small, the AP-diff scheme can accurately reproduce the macroscopic density while it cannot reproduce the distribution function of the internal state.
The AP-hyp scheme is valid in the variation of the parameter value $\tau$.
 It can reproduce the hyperbolic limit behavior both in the macroscopic density and the distribution function of the internal state.

\section*{Acknowledgments}
The authors would like to acknowledge partial funding from the Japan-France Integrated action Program PHC SAKURA, Grant number JPJSBP120193219.

%
%

\bibliographystyle{siamplain}

\newpage
\appendix
\section{Computational Procedure of the AP scheme}\label{app_ap}
The well-balanced and asymptotic preserving (AP) scheme (\ref{upwdiff})-(\ref{eq:pdiff2}) is implemented efficiently by the following procedure.
In the first step (\ref{upwdiff}), $p^{\pm,n+\frac 12}_{i,k}$ for $k\in\{-K,\cdots,-1\}$ are calculated by using the only lower tridiagonal matrix as
\[
p^{\pm,n+\frac 12}_{i,-K}=
p^{\pm,n}_{i,-K}/\left(1+\frac{\Delta t G}{\eps \Delta y}\right),
\]
and, for $k\in\{-K+1,\cdots,-1\}$,
\[
p^{\pm,n+\frac{1}{2}}_{i,k}
=\frac{\eps \Delta y p_{i,k}^{\pm,n}
-\Delta t y_{k-1} p_{i,k-1}^{\pm,n+\frac{1}{2}}}{\eps \Delta y- \Delta t y_k}.
\]
By the same token, $p^{\pm,n+\frac 12}_{i,k}$ for $k\in\{1,\cdots,K\}$ are calculated by using the only upper tridiagonal matrix.
Then, $p_{i,0}^{\pm,n+\frac{1}{2}}$ is calculated as
\[
p_{i,0}^{\pm,n+\frac{1}{2}}+\frac{\Delta t}{\eps \Delta y}
\left(
y_1 p_{i,1}^{\pm,n+\frac{1}{2}}
-y_{-1}p_{i,-1}^{\pm,n+\frac{1}{2}}
\right) = p_{i,0}^{\pm,n}.
\]

In the second step (\ref{eq:pdiff}), by inverting the 2-by-2 matrix, we can explicitly calculate $p^{\pm,n+1}_{i,k}$ for $i\in\{1,\cdots,I-1\}$ and $k\in\{-K,\cdots,K\}$ from
  \begin{align*}
    \left(1+\frac{2\DT}{\eps \DX}\right) p_{i,k}^{+,n+1} =\ & \left(1+\frac{\DT}{\eps\DX}\right) p_{i,k}^{+,n+\frac 12} + \frac{\DT}{\eps \DX} p_{i,k}^{-,n+\frac 12}  \\
    & + \frac{2\DT(\DT+\eps \DX) G}{\eps \DX^2 (2\eps G + \bar{\Lambda}_{k-\frac 12})}(p_{i-1,k-1}^{+,n+\frac 12} - p_{i,k}^{-,n+\frac 12})  \\
    & + \frac{2\DT^2 G}{\eps \DX^2 (2\eps G + \bar{\Lambda}_{k+\frac 12})}(p_{i+1,k+1}^{-,n+\frac 12} - p_{i,k}^{+,n+\frac 12}),
  \end{align*}
  and
  \begin{align*}
    \left(1+\frac{2\DT}{\eps \DX}\right) p_{i,k}^{-,n+1} =\ & \left(1+\frac{\DT}{\eps\DX}\right) p_{i,k}^{-,n+\frac 12} + \frac{\DT}{\eps \DX} p_{i,k}^{+,n+\frac 12}  \\
    & + \frac{2\DT^2 G}{\eps \DX^2 (2\eps G + \bar{\Lambda}_{k-\frac 12})}(p_{i-1,k-1}^{+,n+\frac 12} - p_{i,k}^{-,n+\frac 12}) \\
    & + \frac{2\DT(\DT+\eps \DX) G}{\eps \DX^2 (2\eps G + \bar{\Lambda}_{k+\frac 12})}(p_{i+1,k+1}^{+,n+\frac 12} - p_{i,k}^{-,n+\frac 12}).
  \end{align*}
Hereafter, we set $p_{i,K+1}^{\pm, n+\frac12}=p_{i,-K-1}^{\pm,n+\frac12}=0$, which is equivalent with the non-flux condition $J_{i,K+\frac12}^{\pm,n+\frac12}=J_{i,-K-\frac12}^{\pm,n+\frac12}=0$, for any $i$.
By using the boundary conditions $p^{-,n+1}_{I,k}=p^{+,n+1}_{I,k}$ and $p^{+,n+1}_{0,k}=p^{-,n+1}_{0,k}$ in (\ref{eq:pplusdiff}) and (\ref{eq:pmoinsdiff}), respectively, we can compute
\[
p_{I,k}^{+,n+1}=p_{I,k}^{+,n+\frac 12}
+\frac{2\Delta t G}{\Delta x(2\eps G+\bar \Lambda_{k-\frac 12})}
(p_{I-1,k-1}^{+,n+\frac 12}-p_{I,k}^{-,n+\frac 12}),
\]
\[
p^{-,n+1}_{0,k}=p^{-,n+\frac 12}_{0,k}
+\frac{2\Delta t G}{\Delta x(2\eps G+\bar \Lambda_{k+\frac 12})}
(p_{1,k+1}^{-,n+\frac 12}-p_{0,k}^{+,n+\frac 12}),
\]
for any $k$.

\section{Monte Carlo Method}\label{app_mc}

We have extended a Monte Carlo (MC) code of the classical velocity-jump kinetic equation proposed in \cite{PY,Y} to treat the internal state $y$.
The procedure of the MC method is described below.
Here, we consider \eqref{twovelo} in one-dimensional space $x\in[0,1]$ and the same boundary condition in \eqref{problem}.

\begin{enumerate}
\item The initial states of each MC particle, i.e., the position $x^0_{(\ell)}$, velocity $v^0_{(\ell)}$, and internal state $y^0_{(\ell)}$ at the time step $n=0$ are stochastically determined according to the initial distribution function $p^{\pm,0}(x,y)$.
	Hereafter, the superscript $n$ represents the time step and the subscript $\ell$ represents the index of each MC particle.
\item Given the position, velocity, and internal state of the $l$th MC particle at time step $n$, the position and internal state of the MC particle is advanced in the time-step size $\Delta t$ as
\[
x^{n+1}_{(\ell)}=x^n_{(\ell)} +v^n_{(\ell)}\Delta t,
\]
\[
y^{n+1}_{(\ell)}=\frac{\tau y^n_{(\ell)}+\Delta t Gv^n_{(\ell)}}{\tau+\Delta t}.
\]
\item The particle which moves beyond the boundary at $x=0$ (or $x=1$), say the $l'$th particle at $x=x^{n+1}_{(\ell')}<0$ (or $x=x^{n+1}_{(\ell')}>1$), is relocated at $x=-x^{n+1}_{(\ell')}$ (or $x=2-x^{n+1}_{(\ell')}$) with changing the sign of the velocity as $v=-v^{n+1}_{(\ell')}$.
  This process corresponds to the no-flux boundary conditions.
\item Tumbling of each particle is decided by the probability $\lambda_0\Lambda(y^{n+1}_{(\ell)})\Delta t/2$. The particles, which are decided to make tumble, change the sign of the velocities, $v_{(\ell)}^{n+1}=-v_{(\ell)}^n$, while other particles stay the velocities unchanged.
\item Return to the second process (2).

\end{enumerate}

The macroscopic population density in each lattice site $x\in[i\Delta x,(i+1)\Delta x]$ ($i=0,\cdots,I-1$) is calculated as
\[
\rho^n_i=\frac{1}{N_p\Delta x}\sum_{\ell=1}^{N_p}\int_{i\Delta x}^{(i+1)\Delta x}\delta(x-x^n_{(\ell)})dx,
\]
where $N_p$ is the total particle number.

In \S\ref{Sec_Num}, the time step size $\Delta t= 10^{-4}$ and the total particle number $N_p\simeq 1.2\times 10^6$ are used except the cases for $\lambda_0=10^4$.
For $\lambda_0=10^4$, we set $\Delta t=1\times 10^{-5}$.
%

\section{Note on the mass conservation for the AP-diff scheme}\label{mass_cons}
 
In this part, we consider the mass conservation for the AP-diff scheme.
\begin{lemma}\label{conservation}
  Let us consider the AP-diff scheme \eqref{upwdiff}--\eqref{eq:pdiff} when $G>0$ (the case $G<0$ being similar) complemented with the no-flux boundary conditions 
$J_{i,K+\frac12}^{\pm,n+\frac12}=J_{i,-K-\frac12}^{\pm,n+\frac12}=0$ which are equivalent to
\begin{equation}\label{eqpK}
p_{i,K+1}^{\pm,n+\frac12}=p_{i,-K-1}^{\pm,n+\frac12}=0,
\end{equation}
and with reflection boundary conditions
\begin{equation}\label{eqbc}
p_{0,k}^{+,n}=p_{0,k}^{-,n},
\qquad
p_{I,k}^{-,n}=p_{I,k}^{+,n}.
\end{equation}

Then, if we assume moreover that $p_{i,K}^{+,n}=p_{i,-K}^{-,n}=0$ for any $i\in \{0,\ldots,I\}$, we get that the AP-diff scheme is conservative.
\end{lemma}

\begin{proof}
We define the local density by
\begin{equation*}
  \rho_i=\sum_{k=-K}^K\left(p_{i,k}^++p_{i,k}^-\right)
\end{equation*}
and the total mass $M$, calculated by the trapezoidal law,
\begin{equation*}
M=\frac{\rho_0+\rho_I}{2}+\sum_{i=1}^{I-1}\rho_i.
\end{equation*}



Let 
$\bar p_i^{\pm,n}=\sum_{k=-K}^Kp_{i,k}^{\pm,n}$, and 
$\bar p_i^{\pm,n+\frac12}=\sum_{k=-K}^Kp_{i,k}^{\pm,n+\frac12}$,
we obtain, from \eqref{upwdiff1} with the no-flux boundary conditions, 
$\bar p_i^{\pm,n+\frac12}=\bar p_i^{\pm,n}.$

By summing up \eqref{eq:pplusdiff} and \eqref{eq:pmoinsdiff} w.r.t. $k$, we obtain, respectively,
\begin{subequations}
\begin{equation}\label{pbarplus}
\begin{split}
	\bar p_{i}^{+,n+1}
	&=\bar p_{i}^{+,n+\frac 12}
-\frac{\DT}{\eps\DX}\left(\bar p_{i}^{+,n+1}-\bar p_{i}^{-,n+1}\right)
+\sum_{k=-K}^K a_{k-\frac12}
\left(
p_{i-1,k-1}^{+,n+\frac12}-p_{i,k}^{-,n+\frac12}
\right)\\
&=\bar p_{i}^{+,n+\frac 12}
-\frac{\DT}{\eps\DX}\left(\bar p_{i}^{+,n+1}-\bar p_{i}^{-,n+1}\right)
+\sum_{k=-K+1}^K a_{k-\frac12}
\left(
p_{i-1,k-1}^{+,n+\frac12}-p_{i,k}^{-,n+\frac12}
\right)  \\
& \quad -a_{-K-\frac12}p_{i,-K}^{-,n+\frac12},
\end{split}
\end{equation}
where $a_{k-\frac12}=\frac{2\DT G}{\DX(2\eps G+\bar \Lambda_{k-\frac12})}$, and
\begin{equation}\label{pbarmoins}
\begin{split}
\bar p_{i}^{-,n+1}
&=\bar p_{i}^{-,n+\frac12}
+\frac{\DT}{\eps\DX}\left(\bar p_{i}^{+,n+1}-\bar p_{i}^{-,n+1}\right)
+\sum_{k=-K}^K a_{k+\frac12}
\left(
p_{i+1,k+1}^{-,n+\frac12}-p_{i,k}^{+,n+\frac12}
\right),\\
&=\bar p_{i}^{-,n+\frac12}
+\frac{\DT}{\eps\DX}\left(\bar p_{i}^{+,n+1}-\bar p_{i}^{-,n+1}\right)
+\sum_{k=-K+1}^{K} a_{k-\frac12}
\left(
p_{i+1,k}^{-,n+\frac12}-p_{i,k-1}^{+,n+\frac12}
\right)  \\
&\quad -a_{K+\frac12}p_{i,K}^{+,n+\frac12},
\end{split}
\end{equation}
\end{subequations}
where we use \eqref{eqpK}.
Thus, by summing up the above equations, we have
\begin{align*}
\rho_i^{n+1} =&\ \rho_i^n+
\sum_{k=-K+1}^K a_{k-\frac12}
\left(
p_{i-1,k-1}^{+,n+\frac12}-p_{i,k-1}^{+,n+\frac12}
+p_{i+1,k}^{-,n+\frac12}-p_{i,k}^{-,n+\frac12}
\right)  \\
&\ -a_{-K-\frac12}p_{i,-K}^{-,n+\frac12}
-a_{K+\frac12}p_{i,K}^{+,n+\frac12},
\end{align*}
and by further summing up for $i=1$ to $i=I-1$, we have
\begin{equation}\label{eqrhosum}
\begin{split}
\sum_{i=1}^{I-1}\rho_i^{n+1}
=\sum_{i=1}^{I-1}\rho_i^n
+\sum_{k=-K+1}^K a_{k-\frac12}
\left(
p_{0,k-1}^{+,n+\frac12}-p_{I-1,k-1}^{+,n+\frac12}
+p_{I,k}^{-,n+\frac12}-p_{1,k}^{-,n+\frac12}
\right)&\\
-\sum_{i=1}^{I-1}\left(
a_{-K-\frac12}p_{i,-K}^{-,n+\frac12}
+a_{K+\frac12}p_{i,K}^{+,n+\frac12}
\right)&.
\end{split}
\end{equation}
On the other hand, $\rho_I$ and $\rho_0$ are calculated, respectively,  
from (\ref{pbarmoins}) and (\ref{pbarplus}) as
\begin{equation}\label{eqrhoI}
	\frac{\rho_I^{n+1}}{2}=\frac{\rho_I^n}{2}
	+\sum_{k=-K+1}^K a_{k-\frac12}
\left(
p_{I-1,k-1}^{+,n+\frac12}-p_{I,k}^{-,n+\frac12}
\right)
-a_{-K-\frac12}p_{I,-K}^{-,n+\frac12},
\end{equation}
\begin{equation}\label{eqrho0}
\frac{\rho_0^{n+1}}{2}=\frac{\rho_0^n}{2}
	+\sum_{k=-K+1}^{K} a_{k-\frac12}
\left(
p_{1,k}^{-,n+\frac12}-p_{0,k-1}^{+,n+\frac12}
\right)
-a_{K+\frac12}p_{0,K}^{+,n+\frac12},
\end{equation}
where we use the reflection condition (\ref{eqbc}).

From (\ref{eqrhosum})--(\ref{eqrho0}), we obtain
\begin{equation*}
M^{n+1}=M^n
-\left(
a_{-K-\frac12}\sum_{i=1}^Ip_{i,-K}^{-,n+\frac12}
+a_{K+\frac12}\sum_{i=0}^{I-1} p_{i,K}^{+,n+\frac12}
\right).
\end{equation*}
Since $a_{k-\frac12}>0$, the condition for the mass conservation reads
$p_{i,K}^{+,n}=p_{i,-K}^{-,n}=0$, for any $i\in\{0,\ldots,I\}$.
\end{proof}

\end{document}